\documentclass[reqno0,12pt]{amsart}
\pdfoutput=1
\usepackage{dsfont}
\usepackage{bbm}
\usepackage[nodate]{datetime}
\usepackage[hmargin=30mm,top=28mm,bottom=28mm,a4paper]{geometry}
\usepackage{color}
\usepackage{subfig}
\usepackage{fancyhdr}
\usepackage{amsmath,amssymb,amsfonts,amsthm}
\usepackage{amstext}
\usepackage{amsmath}
\usepackage{amssymb}
\usepackage{enumerate}
\usepackage{amsbsy}
\usepackage{amsopn}
\usepackage{bbm,amsthm}
\usepackage{amscd}
\usepackage{amsxtra}
\usepackage{todonotes}

\newtheorem{theorem}{Theorem}[section]
\newtheorem*{theorem*}{Theorem}
\newtheorem{lemma}{Lemma}[section]

\theoremstyle{remark}

\newtheorem{remark}{Remark}[section]

\newcommand{\RR}{\mathds{R}}

\newcommand{\EE}{\mathbb{E}}

\DeclareMathOperator{\re}{Re}
\DeclareMathOperator{\im}{Im}

\begin{document}
\title{On the expected number of roots of a random Dirichlet polynomial}
\author{Marco Aymone and Caio Bueno}
\begin{abstract}
Let $T>0$ and consider the random Dirichlet polynomial $S_T(t)=\re\sum_{n\leq T} X_n n^{-1/2-it}$, where $(X_n)_{n}$ are i.i.d. Gaussian random variables with mean $0$ and variance $1$. We prove that the expected number of roots of $S_T(t)$ in the dyadic interval $[T,2T]$, say $\EE N(T)$, is approximately $2/\sqrt{3}$ times the number of zeros of the Riemann $\zeta$ function in the critical strip up to height $T$. Moreover, we also compute the expected number of zeros in the same dyadic interval of the $k$-th derivative of $S_T(t)$. Our proof requires the best upper bounds for the Riemann $\zeta$ function known up to date, and also estimates for the $L^2$ averages of certain Dirichlet polynomials.
\end{abstract}

\maketitle

\section{Introduction.}
\subsection{Main result and background} A well know fact (for example, \cite{montgomery_livro}, pg. 454) is that the number of zeros of the Riemann $\zeta$ function in the critical strip $0<\re s <1$ up to (imaginary) height $T$, say $N_\zeta(T)$, is such that
$$N_\zeta(T)=\frac{T}{2\pi}\log\left(\frac{T}{2\pi}\right)-\frac{T}{2\pi}+O(\log(T)),$$
where here we employ the standard notation $O$, $o$, $\ll$ and recall them at Section \ref{secao notacao}.

The Riemann hypothesis is the statement that all zeros of $\zeta$ in the critical strip $0<\re s<1$ are aligned in the vertical line $\re s=1/2$. Since the statement of this hypothesis in 1859, this has motivated a deep study of the Riemann $\zeta$ function on the critical line: $t\mapsto\zeta(1/2+it)$.

One way to approximate $\zeta(1/2+it)$ for $t$ in a dyadic interval $[T,2T]$ is by means of the approximate formula
\begin{equation}\label{equacao approximate zeta}
\zeta(1/2+it)=\sum_{n\leq T}\frac{1}{n^{1/2+it}}+O\left(\frac{1}{\sqrt{T}}\right),
\end{equation}
where the estimate above is uniform in $t\in[T,6.28 T]$ for all sufficiently large $T>0$, see \cite{tenenbaumlivro} pg. 236.

Therefore, it is natural to investigate the zeros of the Dirichlet polynomial 
$$\sum_{n\leq T}\frac{1}{n^{1/2+it}},$$
and motivated by this, here we introduce a random version of it
$$\sum_{n\leq T}\frac{X_n}{n^{1/2+it}},$$
where $(X_n)_n$ are i.i.d. Gaussian random variables with mean $0$ and variance $1$. 

A similar approach \cite{aymone_JLMS} involving partial sums of random multiplicative functions was considered by Heap, Zhao and the first author in order to model the maximal value of $\zeta$ on the critical line in the interval $[T,2T]$. Moreover, the probabilistic approach to the Riemann $\zeta$ function on the critical line has an extensive literature. Here we refer to this work \cite{bourgade_RMT} of Bourgade and Yor (and the references therein) for the many similarities found between $\zeta$ and Random Matrix Theory.

In Probability theory, the study of real roots of random functions is classical and dates back to 1938 since the seminal work of Littlewood and Offord on random polynomials with certain distributions \cite{Littlewood_random_pol_1}, \cite{Littlewood_random_pol_3}, \cite{littlewood_random_int_1}, \cite{littlewood_random_int_2}, by Kac in the Gaussian case \cite{kac_random_poly}, by Erd\H{o}s and Offord in the Rademacher case \cite{erdos_random_poly}, and by Ibragimov and Maslova for more general distributions \cite{ibragimov_random_poly_1}, \cite{ibragimov_random_poly_2}. Here we refer to the following papers for the state of the art in this field: \cite{vu_repulsion} and \cite{oanh_local_universality}, by Do, Nguyen and Vu and by Nguyen and Vu, respectively.

Here we are interested in the number of zeros of the real part of our random Dirichlet polynomial (see Section \ref{secao remarks} for a discussion on this model):
\begin{equation}\label{equacao definicao S_T(t)}
S_T(t)=\re \sum_{n\leq T}\frac{X_n}{n^{1/2+it}}=\sum_{n\leq T}\frac{X_n\cos(t\log n)}{\sqrt{n}}.
\end{equation}

We denote the number of zeros of $S_T(t)$ in the interval $[T,2T]$ by
\begin{equation}\label{equacao definicao N(T)}
N(T):=|\{t\in[T,2T]: S_T(t)=0\}|.
\end{equation}

Let $N_{0,\re \zeta}(T)$ be the number of zeros of $\re\zeta(1/2+it)$ in the interval $[T,2T]$. We can deduce from a result by Garaev (see the Remark at pg. 250 of \cite{garaev_real_part}) that 
$$N_{0,\re \zeta}(T)=\frac{T}{2\pi}\log(T)+N_{0,\zeta}(T)+O(T),$$
where $N_{0,\zeta}(T)$ is the number of zeros of $\zeta(1/2+it)$ in the interval $[T,2T]$.

Therefore, if the Riemann hypothesis is true, then 
\begin{equation}\label{equacao Nzeta real RH}
N_{0,\re \zeta}(T)=(2+o(1))( N_{\zeta}(2T)-N_{\zeta}(T)).
\end{equation} 
For our random model we prove that the proportion $\EE N(T)/ (N_{\zeta}(2T)-N_{\zeta}(T))$ is smaller than $2+o(1)$.
\begin{theorem}\label{teorema esperanca} Let $N(T)$ be as defined in \eqref{equacao definicao N(T)}. Then
$$\EE N(T)=\frac{1}{\pi\sqrt{3}} T \log T -\frac{\gamma}{2\pi\sqrt{3}}T+O\left(\frac{T}{\log T}\right),$$
where $\gamma$ is the Euler-Mascheroni constant.
\end{theorem}
In particular, we have that
$$\EE N(T)=\left(\frac{2}{\sqrt{3}}+o(1)\right)( N_{\zeta}(2T)-N_{\zeta}(T)),$$
and this is interesting to compare with \eqref{equacao Nzeta real RH} under the Riemann hypothesis.

\begin{remark} We stress that the same result of Theorem \ref{teorema esperanca} (with same constants) holds if we replace the real part of $\sum_{n\leq T}X_nn^{-1/2-it}$ by its imaginary part. Moreover, a variant of Theorem \ref{teorema esperanca} should hold for other distributions, that is, the hypothesis that $(X_n)_n$ are standard Gaussian random variables should not be too restrictive if one takes into account the universality results of Nguyen and Vu \cite{oanh_local_universality}. Therefore, one should obtain the same leading term $\EE N(T)=\frac{1+o(1)}{\pi\sqrt{3}} T \log T$ if the Gaussians are replaced by independent Rademacher $\pm 1$ random variables, for example. However, the change of distribution of $(X_n)_n$ might affect the lower order terms of $\EE N(T)$.
\end{remark}

\subsection{Derivatives of $S_T(t)$} For a positive integer $k$, a result due to Berndt \cite{berndt_zeros_derivative} states that the number of zeros of the $k$-th derivative of the Riemann $\zeta$ function, $\zeta^{(k)}(\sigma+it)$, with $0<t<T$ is given by
$$\frac{1}{2\pi}T\log T-\left(\frac{1+\log 4\pi}{2\pi}\right)T+O(\log T).$$

See also the paper \cite{irma_zeros_derivative} by Ge and Suriajaya for conditional estimates for this number of zeros.

The approximation \eqref{equacao approximate zeta} is still valid for the derivatives of the Riemann $\zeta$ function, but in this case we have uniformly for all $t\in[T,6.28T]$ (see Lemma 4 of \cite{daodao_BLMS})
$$\zeta^{(k)}(1/2+it)=(-1)^k\sum_{n\leq T}\frac{(\log n)^k}{n^{1/2+it}}+O_{k,\epsilon}\left(\frac{1}{T^{1/2-\epsilon}}\right),$$
where $\epsilon>0$ and the notation $O_{k,\epsilon}$ means that the implied constant depends only on $k$ and $\epsilon$.

Motivated by this, we consider 
\begin{equation}\label{equacao N zeros derivada}
N^{(k)}(T):=|\{t\in[T,2T]:S_T^{(k)}(t)=0\}|,
\end{equation}
where 
$$S_T^{(k)}(t)=\frac{d^k}{dt^{k}}S_T(t)=\re\,(-i)^{k}\sum_{n\leq T}\frac{X_n(\log n)^k}{n^{1/2+it}}.$$

Our next result generalizes Theorem \ref{teorema esperanca}.
\begin{theorem}\label{teorema zeros derivada} Let $N^{(k)}(T)$ be as in \eqref{equacao N zeros derivada}. Then
$$ \EE N^{(k)}(T)=\frac{1}{\pi}\sqrt{\frac{2k+1}{2k+3}} T \log T -\frac{\gamma_{2k}}{2\pi}\sqrt{\frac{(2k+1)^3}{2k+3}}\frac{T}{(\log T)^{2k}}+O\left(\frac{T}{(\log T)^{2k+1}}\right),$$
where $\gamma_k$ is the $k$-th Stieltjes constant.
\end{theorem}

\subsection{Structure of the paper} We conclude this introduction by briefly speaking about the next sections. We begin at Section \ref{secao notacao} with the main asymptotic notation used in this paper. Then we proceed in Section \ref{secao preliminares} with the main tools used in the proof of the main results, this include tools from Probability, Analytic Number Theory and Analysis. We prove Theorem \ref{teorema esperanca} at Section \ref{secao prova teorema 1}. In Section \ref{secao prova teorema 2} we only indicate the main modifications in the proof of Theorem \ref{teorema esperanca} to obtain Theorem \ref{teorema zeros derivada}. We prefer to do in this way to avoid overloaded notation, so that the ideas become more clear. We conclude this paper with some remarks at Section \ref{secao remarks}.

\section{Notation}\label{secao notacao}

We use the standard notation: 
\begin{enumerate}
	\item $f(x)\ll g(x)$ or equivalently $f(x)=O(g(x))$;
	\item $f(x)=o(g(x))$;
	\item $f(x)\sim g(x)$.
\end{enumerate}
The case (1) is used whenever there exists a constant $C>0$ such that $|f(x)|\leq C |g(x)|$, for all $x$ in a set of numbers. This set of numbers when not specified is the real interval $[L,\infty]$, for some $L>0$, but also there are instances where this set can accumulate at the right or at the left of a given real number, or at complex number. Sometimes we also employ the notation $\ll_\epsilon$ or $O_\epsilon$ to indicate that the implied constant may depends on $\epsilon$. 

In case (2), we mean that $\lim_{x}f(x)/g(x)=0$. When not specified, this limit is as $x\to \infty$ but also can be as $x$ approaches any complex number in a specific direction. 

In case (3), we mean that $f(x)=(1+o(1))g(x)$.

\section{Preliminaries}\label{secao preliminares}

\subsection{The main tool from Probability Theory} The proof of our results starts with the celebrated formula of Edelman and Kostlan \cite{edelman_roots_random_poly}: If 
$$v(t)=(v_1(t),v_2(t),...,v_N(t))$$
is a vector of twice differentiable functions $v_j:\RR\to\RR$, $1\leq j\leq N$, and if $X_1,...,X_N$ are i.i.d. Gaussian random variables with mean $0$ and variance $1$, 
then, for any interval $I\subset \RR$, the expected number of solutions $t\in I$ of the equation
$$X_1v_1(t)+...+X_Nv_N(t)=0$$
equals to
\begin{equation}\label{equacao formula zeros edelman and kostlan}
\frac{1}{\pi}\int_I\sqrt{\frac{\partial^2}{\partial x\partial y}\log v(x)\cdot v(y)\bigg{|}_{x=y=t}}dt.
\end{equation}

\subsection{Estimates for the Riemann $\zeta$ function and its derivatives} Obtaining estimates for the Riemann $\zeta$ function on the critical strip is a classical and an important problem. In our proofs we will use the bound
$$\zeta(\sigma+it)\ll (1+|t|^{100(1-\sigma)^{3/2}})(\log|t|)^{2/3}$$
that holds uniformly in the region $\sigma\geq 0$ and $|t|\geq 2$, see \cite{titchmarsh-theory-rzf} pg. 135. This bound is obtained upon the famous method of Vinogradov--Korobov on exponential sums and is due to Richert \cite{richert_zeta_bound}.

Therefore, by the estimate above, if $1-\sigma\leq \frac{1}{100(\log |t|)^{2/3}}$, then
\begin{equation}\label{equacao estimativa zeta}
\zeta(\sigma+it)\ll (\log|t|)^{2/3}.
\end{equation}

This bound allow us to prove the following result regarding the derivatives of $\zeta$.

\begin{lemma}\label{lemma derivadas da zeta}
    Let $k\geq 1$ be a fixed integer and $\zeta^{(k)}$ denote the $k$-th derivative of the Riemann zeta function. Then
        \begin{align*}
            |\zeta^{(k)}(1+it)|\ll_k (\log t)^{\frac{2}{3}(k+1)},
        \end{align*}
    for $t\geq 2$.
\end{lemma}
\begin{proof}
    Our aim is to apply Cauchy's integral formula on a suitable curve and bound the $k$-th derivative $\zeta^{(k)}$ at the point $1+it$, by using \eqref{equacao estimativa zeta}.

Let $\varepsilon(t)=\frac{1}{100(\log t)^{2/3}}$ and $\mathcal{C}$ be a circle centered at $1+it$ and with radius\footnote{A small remark about the radius $\frac{\varepsilon(t)}{4}$ is that, since we want to use the estimate for the Riemann zeta function, we need to take the circle $\mathcal{C}$ entirely inside the region such that this estimate remains true. The curve $1-\frac{1}{100(\log t)^{2/3}}$ monotonically approaches the line $1+it$ as $t\to +\infty$, so we can't simply use $\varepsilon(t)$ as the radius. It turns out that only a small correction is needed and multiplying by the constant $\frac{1}{4}$ is enough to keep $\mathcal{C}$ strictly to the right of this curve.} $\frac{\varepsilon(t)}{4}$.

    By Cauchy's integral formula,
        \begin{align*}
            \zeta^{(k)}(1+it)=\frac{k!}{2\pi i}\int_{\mathcal{C}}\frac{\zeta(s)}{(s-(1+it))^{k+1}} \,ds.
        \end{align*}

    Therefore, estimating the integral on the right hand side, we have
        \begin{align*}
            \bigg|\int_{\mathcal{C}}\frac{\zeta(s)}{(s-(1+it))^{k+1}}\,ds\bigg| \ll_k (\log t)^{(k+1)\frac{2}{3}},
        \end{align*}
    where we used the fact that $|s-(1+it)|$ is exactly the radius of the circle and also the bound for the Riemann zeta function discussed earlier. Thus, our result follows.

\end{proof}

Another result we will need and already presented in the introduction is the following.

\begin{lemma}[Yang \cite{daodao_BLMS}, p. 8, Lemma 4, and \cite{tenenbaumlivro}, pg. 236]\label{lemma daodao}
    Let $\sigma_0\in(0,1)$ be fixed. If $T$ is sufficiently large, then uniformly for $\epsilon>0$, $t\in [T,6.28T]$, $\sigma \in [\sigma_0+\epsilon,\infty)$ and all integers $k\geq 0$, we have
        \begin{align*}
            (-1)^k\zeta^{(k)}(\sigma+it)=\sum_{n\leq T}\frac{(\log n)^k}{n^{\sigma+it}}+O_{\sigma_0}\bigg(\frac{k! T^{-\sigma+\epsilon}}{\epsilon^k}\bigg).
        \end{align*}
\end{lemma}

\subsection{The main tool from Fourier Analysis} In the course of our proof we will need the following estimate on $L^2$ averages of Dirichlet polynomials: If $a_1,...,a_N$ are complex numbers, then
\begin{equation}\label{equacao lemma ivic}
\int_{0}^T\left|\sum_{n\leq N}a_nn^{it}\right|^2dt=T\sum_{n\leq N}|a_n|^2+O\left(\sum_{n\leq N}n|a_n|^2\right).
\end{equation}

A proof of this result can be found at \cite{ivic_livro} pg. 130. A straightforward application of this result lead us to the following result.

\begin{lemma}\label{lemma L^2 average dirichlet polynomials} For any non-negative integer $k$, we have that
$$\int_{0}^{2T}\left|\sum_{n\leq T}\frac{(\log n)^k}{n}n^{it}\right|^2dt\ll_kT.$$
\end{lemma}

\subsection{Some quick estimates from Real analysis} We will need the following results.
\begin{lemma}\label{lemma serie de taylor} If $A$ and $B$ are real numbers such that $|B|<A$, then
\begin{align*}
\frac{1}{A+B}&=\frac{1}{A}-\frac{B}{A^2}+O\left(\frac{B^2}{A^3}\right)\\
\sqrt{A+B}&=\sqrt{A}+\frac{B}{2\sqrt{A}}+O\left(\frac{B^2}{A^{3/2}}\right).
\end{align*}
\end{lemma}
\begin{proof} The first estimate is a direct application of the estimate $\frac{1}{1+x}=1-x+O(x^2)$, $|x|<1$, and the second an application of $\sqrt{1+x}=1+x/2+O(x^2)$, with $|x|<1$.
\end{proof}
\begin{lemma}\label{lema stieltjes constants} For any $T\geq 2$ and an integer $k\geq 0$, 
$$\sum_{n\leq T}\frac{(\log n)^k}{n}=\frac{(\log T)^{k+1}}{k+1}+\gamma_k+O_k\left(\frac{(\log T)^k}{T}\right).$$
The constant $\gamma_k$ is called $k$-th Stieltjes constant.
\end{lemma}
\begin{proof} Let $[x]$ and $\{x\}$ denote the integer and fractional part of $x$ respectively. Then we can write the following sum as a Riemann-Stieltjes integral:
$$\sum_{n\leq T}\frac{(\log n)^k}{n}=\int_1^T\frac{ (\log x)^k}{x}dx-\int_{1^-}^T\frac{ (\log x)^k}{x}d\{x\},$$
where we used that $d[x]=dx-d\{x\}$. The first integral above gives the main term $\frac{(\log T)^{k+1}}{k+1}$. In the second integral we use integration by parts:
\begin{align*}
-\int_{1^-}^T \frac{(\log x)^k}{x}d\{x\}&=\int_1^T\frac{(\log x)^k-k(\log x)^{k-1}}{x^2}dx+O\left(\frac{(\log T)^k}{T}\right)\\
&=c-\int_T^\infty\frac{ (\log x)^k-k(\log x)^{k-1}}{x^2}dx+O\left(\frac{(\log T)^k}{T}\right)\\
&:=c-J_T+O\left(\frac{(\log T)^k}{T}\right),
\end{align*}
where $c$ is a constant due to the fact that the function inside the integral $J_T$ is integrable in the interval $[1,\infty)$.
To estimate $J_T$, we divide the interval $[T,\infty)$ in subintervals $[T,3T)\cup [3T,3^2T)\cup...$. In each subinterval we have that 
$$\int_{3^jT}^{3^{j+1}T}\frac{ (\log x)^k-k(\log x)^{k-1}}{x^2}dx\ll \frac{j(\log T)^k}{3^jT}.$$
Hence, by summing the above estimate in $j$, we obtain the target result.
\end{proof}

\begin{lemma}\label{lema desigualdade holder discreta} Let $x_1,...,x_R$ be real numbers. Then, for any $k\geq 1$,
$$|(x_1+...+x_R)^k|\leq R^{k-1}(|x_1|^k+...+|x_R|^k).$$
\end{lemma}
\begin{proof} The proof is a direct application of H\"older's inequality with conjugated exponents $p^{-1}+q^{-1}=1$, with $p=k$:
$$|x_1+...+x_R|\leq R^{1-1/k}(|x_1|^k+...+|x_R|^k)^{1/k}.$$
\end{proof}

\section{Proof of Theorem \ref{teorema esperanca}}\label{secao prova teorema 1}
We begin the proof by applying the formula of Edelman and Kostlan \eqref{equacao formula zeros edelman and kostlan} in our context.
\begin{lemma}\label{lemma formula zeros nosso caso}
    Let $(X_n)_n$ be i.i.d. Gaussian random variables with mean $0$ and variance $1$ and $T>0$ be fixed. Then the expected number of real zeros of
        \begin{align*}
            S_T(t)=\sum_{n\leq T}\frac{X_n \cos(t \log n)}{\sqrt{n}},
        \end{align*}
    on an arbitrary interval $I$, is
        \begin{align*}
            \frac{1}{\pi} \int_{I}\bigg(\frac{\sum_{n\leq T}\frac{(\log n)^2}{2n}-\sum_{n\leq T}\frac{\cos(2t \log n)(\log n)^2}{2n}}{\sum_{n\leq T}\frac{1}{2n}+\sum_{n\leq T}\frac{\cos(2t \log n)}{2n}}-\bigg(\frac{\sum_{n\leq T}-\frac{\sin(2t \log n)(\log n)}{2n}}{\sum_{n\leq T}\frac{1}{2n}+\sum_{n\leq T}\frac{\cos(2t \log n)}{2n}}\bigg)^2\bigg)^{\frac{1}{2}} \,dt.
        \end{align*}
\end{lemma}
\begin{proof}
    We let $v(t)=(f_1(t),\dots, f_{[T]}(t))$ where $f_n(t)=\frac{\cos(t\log n)}{\sqrt{n}}$ and apply the result by Edelman and Kostlan \eqref{equacao formula zeros edelman and kostlan}.

    We have,
        \begin{align}\label{equacao edelman kostlan}
            v(x)\cdot v(y)=\sum_{n\leq T}f_n(x)f_n(y)=\sum_{n\leq T}\frac{\cos(x \log n)\cos(y \log n)}{n}.
        \end{align}

    Now we compute the logarithmic derivative with respect to $y$ to obtain
        \begin{align*}
            \frac{\partial}{\partial y}\log\bigg(\sum_{n\leq T}f_n(x)f_n(y)\bigg)&=\frac{\sum_{n\leq T}f_n(x)f_n'(y)}{\sum_{n\leq T}f_n(x)f_n(y)}.
        \end{align*}
        
    Computing the derivative with respect to $x$, we get the following
        \begin{align}\label{integrando edelman kostlan}
            \frac{\partial^2}{\partial x\partial y}\log(v(x)\cdot v(y))&=\frac{\partial}{\partial x}\bigg(\frac{\sum_{n\leq T}f_n(x)f_n'(y)}{\sum_{n\leq T}f_n(x)f_n(y)}\bigg) \nonumber\\
            &=\frac{\sum_{n\leq T}f_n'(x)f_n'(y)}{\sum_{n\leq T}f_n(x)f_n(y)}-\frac{\sum_{n\leq T}f_n(x)f_n'(y)\sum_{n\leq T}f_n'(x)f_n(y)}{(\sum_{n\leq T}f_n(x)f_n(y))^2}.
        \end{align}

    Finally we set $x=y=t$ and use trigonometric identities to rewrite the above result. Since the derivative of our function $f_n$ is $f_n'(t)=\frac{-\sin(t\log n)(\log n)}{\sqrt{n}}$, by \eqref{equacao edelman kostlan} the first term can be written as
        \begin{align*}
            \frac{\sum_{n\leq T}f_n'(t)f_n'(t)}{\sum_{n\leq T}f_n(t)f_n(t)}&=\frac{\sum_{n\leq T}\frac{\sin^2(t \log n)(\log n)^2}{n}}{\sum_{n\leq T}\frac{\cos^2(t \log n)}{n}}=\frac{\sum_{n\leq T}\frac{(\log n)^2}{2n}-\sum_{n\leq T}\frac{\cos(2t \log n)(\log n)^2}{2n}}{\sum_{n\leq T}\frac{1}{2n}+\sum_{n\leq T}\frac{\cos(2t \log n)}{2n}},
        \end{align*}
    where in the second equality we used that $\sin^2(\theta)=\frac{1-\cos(2\theta)}{2}$ and $\cos^2(\theta)=\frac{1+\cos(2\theta)}{2}$.

    For the second term of \eqref{integrando edelman kostlan}, we have
        \begin{align*}
            \frac{\sum_{n\leq T}f_n(t)f_n'(t)\sum_{n\leq T}f_n'(t)f_n(t)}{(\sum_{n\leq T}f_n(t)f_n(t))^2}&=\bigg(\frac{\sum_{n\leq T}-\frac{\sin(t\log n)\cos(t\log n)(\log n)}{n}}{\sum_{n\leq T}\frac{\cos^2(t\log n)}{n}}\bigg)^2\\
            &=\bigg(\frac{\sum_{n\leq T}-\frac{\sin(2t \log n)(\log n)}{2n}}{\sum_{n\leq T}\frac{1}{2n}+\sum_{n\leq T}\frac{\cos(2t \log n)}{2n}} \bigg)^2,
        \end{align*}
    where here we used that $2\sin(\theta)\cos(\theta)=\sin(2\theta)$. Putting these two results together and applying the Edelman and Kostlan formula we obtain our Lemma.
\end{proof}

Now we continue with

\begin{proof}[Proof of Theorem \ref{teorema esperanca}] Define 
$$u_T(t):=\sum_{n\leq T}\frac{\cos(t\log n)}{n}.$$

By Lemmas \ref{lema stieltjes constants} and \ref{lemma formula zeros nosso caso}, we have that
\begin{align*}
\EE N(T)&=\frac{1}{\pi}\int_{T}^{2T}\left(\frac{\frac{(\log T)^3}{3}+\gamma_2+O\left(\frac{(\log T)^2}{T}\right)+u_T''(2t)}{\log T+\gamma+O\left(\frac{1}{T}\right)+u_T(2t)}-(1+o(1))\frac{(u_T'(2t))^2}{(\log T)^2}\right)^{1/2}dt\\
&:=\frac{1}{\pi}\int_{T}^{2T}\sqrt{A/B-C^2}dt.
\end{align*}

Now we will rewrite 
$$A/B=\frac{\frac{(\log T)^3}{3}+\gamma_2+O\left(\frac{(\log T)^2}{T}\right)+u_T''(2t)}{\log T+\gamma+O\left(\frac{1}{T}\right)+u_T(2t)}$$ 
in a more tractable manner. We begin by defining
\begin{align*}
x&:=\frac{\gamma}{\log T}+O\left(\frac{1}{T\log T}\right)+\frac{u_T(2t)}{\log T},\\
y&:=\frac{3\gamma_2}{(\log T)^3}+O\left(\frac{1}{T\log T}\right)+\frac{3u_T''(2t)}{(\log T)^3},
\end{align*}
so that
$$A/B=\frac{(\log T)^2}{3}\frac{1+y}{1+x}.$$

By estimate \eqref{equacao estimativa zeta}, Lemma \ref{lemma derivadas da zeta} and Lemma \ref{lemma daodao}, for $j=0,1,2$ we have that in the range $t\in[T,2T]$,
\begin{equation}\label{equacao estimativas u''e u'}
|u_T^{(j)}(2t)| \ll |\zeta^{(j)}(1+2it)|\ll (\log T)^{2(j+1)/3}.
\end{equation}

Therefore $x=O((\log T)^{-1/3})$, and hence, by Lemma \ref{lemma serie de taylor}
$$A/B=\frac{(\log T)^2}{3}\frac{1+y}{1+x}=\frac{(\log T)^2}{3}(1-x+y-xy+O(x^2+|y|x^2)).$$
Now we define $z=3C^2/(\log T)^2$ and write
$$A/B-C^2=\frac{(\log T)^2}{3}(1-x+y-xy-z+O(x^2+|y|x^2)).$$

By \eqref{equacao estimativas u''e u'},
$$y\ll (\log T)^{-1},\;z\ll (\log T)^{-4/3}.$$

Therefore, $w:=-x+y-xy-z+O(x^2+|y|x^2)$ is $o(1)$, and hence, by Lemma \ref{lemma serie de taylor}
$$\sqrt{A/B-C^2}=\frac{\log T}{\sqrt{3}}(1+w/2+O(w^2)).$$

Now we can write
$$\EE N(T)=\frac{1}{\pi \sqrt{3}}T\log T+\frac{1}{2\pi \sqrt{3}}\log T\int_T^{2T}wdt+O\left(\log T\int_T^{2T}w^2dt\right).$$

Now we will break the integral of $w$ in several pieces.

\noindent Step 1: $-\int_{T}^{2T}xdt=-\frac{\gamma T}{\log T}+O\left(\frac{\log\log T}{\log T}\right)$.

We recall that
$$x=\frac{\gamma}{\log T}+O\left(\frac{1}{T\log T}\right)+\frac{u_T(2t)}{\log T}.$$
Observe that
\begin{align*}
\int_T^{2T}u_T(2t)dt&=\sum_{n\leq T}\frac{1}{2n\log n}(\sin(4T\log n)-\sin(2T\log N))\\
&\ll\sum_{n\leq T}\frac{1}{n\log n}\\
&\ll \log\log T.
\end{align*}
This justifies step 1.

\noindent Step 2: $\int_{T}^{2T}ydt\ll \frac{T}{(\log T)^3}$.

We recall that
$$y=\frac{3\gamma_2}{(\log T)^3}+O\left(\frac{1}{T\log T}\right)+\frac{3u_T''(2t)}{(\log T)^3}.$$
Now, as discussed earlier, by \eqref{equacao estimativas u''e u'} we have that,
$$\int_{T}^{2T}u_T''(2t)dt\ll |u_T'(2T)|+|u_T'(T)|\ll (\log T)^{4/3}.$$
This justifies step 2.

\noindent Step 3: $\int_{T}^{2T}xydt\ll \frac{T}{(\log T)^4}$.

We have that
\begin{align*}
xy &=\left(O\left(\frac{1}{\log T}\right)+\frac{u_T(2t)}{\log T}\right)\left(O\left(\frac{1}{(\log T)^3}\right)-\frac{3u_T''(2t)}{(\log T)^3}\right)\\
& =\frac{1}{(\log T)^4}O(1+|u_T(2t)|+|u_T''(2t)|+|u_T(2t)||u_T''(2t)|).
\end{align*}
Now, by the Cauchy-Schwarz inequality and Lemma \ref{lemma L^2 average dirichlet polynomials}, we have that
$$\int_{T}^{2T}|u_T(2t)|dt\leq \sqrt{T}\left(\int_{0}^{2T}|u_T(2t)|^2\right)^{1/2}\ll T,$$
and similarly
\begin{align*}
&\int_{T}^{2T}|u_T''(2t)|dt\ll T,\\
&\int_{T}^{2T}|u_T(2t)||u_T''(2t)|dt\ll T.
\end{align*}
This justifies step 3.

Now, similarly to step 3, we obtain\\
\noindent Step 4: $\int_{T}^{2T}zdt\ll \frac{T}{(\log T)^4}$.

\noindent Step 5: $\int_{T}^{2T}x^2dx\ll \frac{T}{(\log T)^2}$.

By Lemma \ref{lema desigualdade holder discreta} we have that 
$$x^2=\frac{1}{(\log T)^2}O(1+u_T(t)^2),$$
and hence, similarly to Step 3 we obtain Step 5.

\noindent Step 6: $\int_{T}^{2T}|y|x^2dt\ll \frac{T}{(\log T)^{5-4/3}}$.

We have that
\begin{align*}
|yx^2|&=\frac{1}{(\log T)^5}O((1+|u_T''(2t)|)(1+|u_T(2t)|^2))\\
&=\frac{1}{(\log T)^5}O((1+|u_T''(2t)|+|u_T(2t)|^2+|u_T''(2t)|\cdot|u_T(2t)|^2)).
\end{align*}

Thus, the new kind of integral that we have to deal with it is of $|u_T''(2t)|\cdot|u_T(2t)|^2$, the others can be treated similarly to step 3. To proceed with it, we use the pointwise bound
$$|u_T''(2t)|\cdot|u_T(2t)|^2\ll (\log T)^{4/3}|u_T''(2t)|.$$
Therefore
$$\int_T^{2T}|u_T''(2t)|\cdot|u_T(2t)|^2dt\ll T(\log T)^{4/3},$$
and this justifies step 6.

We conclude that, by collecting all of these steps, we have that
$$\EE N(T)=\frac{1}{\pi\sqrt{3}} T \log T -\frac{\gamma}{2\pi\sqrt{3}}T+O\left(\frac{T}{\log T}\right)+O\left(\log T\int_T^{2T}w^2dt\right).$$

Thus, it remains to estimate $\int_T^{2T}w^2dt$. By Lemma \ref{lema desigualdade holder discreta} again, we obtain
$$w^2\ll x^2+y^2+x^2y^2+z^2+x^4+y^2x^4.$$
The integral $\int_{T}^{2T}(x^2+y^2+z^2)dt$ can be bounded as in the same way of the steps above. It is possible to show that this integral is $O(T/(\log T)^2)$.

\noindent Step 7: $\int_{T}^{2T}x^4dt\ll \frac{T}{(\log T)^{4-4/3}}$.

By Lemma \ref{lema desigualdade holder discreta} applied with $R=2$ and $k=4$:
$$x^4=\frac{1}{(\log T)^4}O(1+u_T(2t)^4).$$
Now $u_T(2t)^4\ll (\log T)^{4/3}u_T(2t)^2$. Therefore
$$\int_{T}^{2T}u_T(2t)^4dt\ll T(\log T)^{4/3},$$
and this justifies step 4.

\noindent Step 8: $\int_{T}^{2T} x^2y^2dt\ll \frac{T}{(\log T)^4}$.

Similarly to step 7, we can show that
$$\int_{T}^{2T}y^4dt\ll \frac{T}{(\log T)^8}.$$
This combined with Cauchy-Schwarz inequality gives that
$$\int_{T}^{2T} x^2y^2dt\leq\left(\int_{T}^{2T}x^4dt\right)^{1/2}\left(\int_{T}^{2T}y^4dt\right)^{1/2},$$
and this justifies step 8.

\noindent Step 9 (Final): $\int_{T}^{2T} y^2x^4dt\ll \frac{T}{(\log T)^{6-4/3}}$.

Here we use the pointwise bound
$$y^2x^4\ll \frac{x^2y^2}{(\log T)^{2-4/3}},$$
and conclude with the previous step.

Collecting all of these last steps, we conclude that $\log T\int_{T}^{2T}w^2dt\ll \frac{T}{(\log T)^{3-4/3}}$, 
and this completes the proof. \end{proof}

\section{Proof of Theorem \ref{teorema zeros derivada}}\label{secao prova teorema 2}
Now we present the main modifications in the proof of Theorem \ref{teorema esperanca} in order to obtain Theorem \ref{teorema zeros derivada}.

We begin by observing that either
$$S_T^{(k)}(t)=\pm \sum_{n\leq T}\frac{X_n(\log n)^k\cos(t\log n)}{\sqrt{n}},$$
or
$$S_T^{(k)}(t)=\pm \sum_{n\leq T}\frac{X_n(\log n)^k\sin(t\log n)}{\sqrt{n}}.$$

The approach with $\sin$ in place of $\cos$ is the same, so we assume that it holds the first option above. In the same way as in Lemma \ref{lemma formula zeros nosso caso}, we can obtain that
\begin{align*}
            \frac{1}{\pi} \int_{I}\bigg(\frac{\sum_{n\leq T}\frac{(\log n)^{2k+2}}{2n}+\frac{u_T''(2t)}{2}}{\sum_{n\leq T}\frac{(\log n)^{2k}}{2n}+\frac{u_T(2t)}{2}}-\bigg(\frac{\frac{u_T'(2t)}{2}}{\sum_{n\leq T}\frac{(\log n)^{2k}}{2n}+\frac{u_T(2t)}{2}}\bigg)^2\bigg)^{\frac{1}{2}} \,dt,
        \end{align*}
where
$$u_T(t)=\sum_{n\leq T}\frac{(\log n)^{2k}\cos(t\log n)}{n}.$$
Repeating the same procedure and defining the function inside the squareroot above by $A/B-C^2$, by Lemma \ref{lema stieltjes constants}, we have that
$$A/B-C^2=\frac{2k+1}{2k+3}(\log T)^2\left(\frac{1+y}{1+x}-z\right),$$
where
\begin{align*}
x&:=\frac{(2k+1)\gamma_{2k}}{(\log T)^{2k+1}}+O\left(\frac{1}{T\log T}\right)+\frac{(2k+1)u_T(2t)}{(\log T)^{2k+1}},\\
y&:=\frac{(2k+3)\gamma_2}{(\log T)^{2k+3}}+O\left(\frac{1}{T\log T}\right)+\frac{(2k+3)u_T''(2t)}{(\log T)^3},\\
z&:=\frac{(2k+1)(2k+3)u_T'(2t)^2}{(1+o(1))(\log T)^{4k+4}}.
\end{align*}

Now, by Lemmas \ref{lemma derivadas da zeta} and \ref{lemma daodao}, we have that
\begin{align*}
x&\ll(\log T)^{-(2k+1)/3},\\
y&\ll(\log T)^{-(2k+3)/3},\\
z&\ll(\log T)^{-4(k+1)/3},
\end{align*}
and these  bounds are stronger than as in the case $k=0$ of the proof of Theorem \ref{teorema esperanca}.

Now, by defining $w$ in the exactly same way as before $w:=-x+y-xy-z+O(x^2+|y|x^2)$, we have that
$$\EE N^{(k)}(T)=\frac{1}{\pi}\sqrt{\frac{2k+1}{2k+3}}T\log T+\frac{1}{2\pi}\sqrt{\frac{2k+1}{2k+3}}\log T\int_T^{2T}wdt+O\left(\log T\int_T^{2T}w^2dt\right).$$

The main contribution of $\int_{T}^{2T}wdt$ comes from $\int_{T}^{2T}xdt$ which gives the second main term, and by proceeding in the same manner as in steps 1-9 of the proof of Theorem \ref{teorema esperanca} we can complete the argument. 

\section{Concluding remarks}\label{secao remarks}
\subsection{Further discussion on the presented model} One interesting question that needs an answer is about the number of zeros of $\sum_{n\leq T}\frac{X_n}{n^{1/2+it}}$, not necessarily of the real or imaginary part isolated. Notice that in this case, for
$$\sum_{n\leq T}\frac{X_n}{n^{1/2+it}}=0$$
one needs to vanish the real and the imaginary part of this sum simultaneously. We did not find in the literature how to handle the case in which the real and the imaginary part have a non-trivial correlation. Indeed, the covariance is given by
$$\EE \re \sum_{n\leq T}\frac{X_n}{n^{1/2+it}} \im \sum_{m\leq T}\frac{X_m}{m^{1/2+it}}=-\sum_{n\leq T}\frac{\cos(t\log n)\sin(t\log n)}{n}=-\frac{1}{2}\sum_{n\leq T}\frac{\sin(2t\log n)}{n},$$
and the last sum can be seen as an approximation to $\frac{1}{2}\im \zeta(1-2it)$. It remains an open question to compute the expected number of zeros in this case.

In contrast, if $(X_n)_n$ are complex Gaussian random variables, in the sense that the real and the imaginary part of $X_n$ are independent and have Gaussian distribution with mean $0$ and variance $1$, then a formula for the number of zeros of
$$P_T(\sigma+it):=\sum_{n\leq T}\frac{X_n}{n^{\sigma+it}}$$
is known, see for example Theorem 8.1 of \cite{edelman_roots_random_poly}. In this case, to see something relevant on the expected counting of the number of zeros, we should look for zeros on sets of positive two-dimensional Lebesgue measure of $\mathbb{C}$. Therefore, the expected number of zeros over a one-dimensional set (like a line) of $P_T(\sigma+it)$ is zero.

\subsection{Other related models} Another interesting model that one could try to obtain information on the expected number of zeros is the one with a random multiplicative function $f(n)$ in place of $X_n$. A Steinhaus random multiplicative function for example, is obtained by letting over primes $(f(p))_p$ be sequence of i.i.d. random variables with uniform distribution over the complex unitary circle, and at the positive integer $n$ factorized as $n=p_1^{a_1}\cdot...\cdot p_k^{a_k}$,
$$f(n)=f(p_1)^{a_1}\cdot...\cdot f(p_k)^{a_k}.$$
In a certain sense, the Steinhaus random multiplicative function $f$ can be seen as model for $n\mapsto n^{it}$, and this approach was used before by Heap, Zhao and the first author in \cite{aymone_JLMS} in order do model the maximal value of $\zeta(1/2+it)$.

\subsection{An ending remark} We conclude this paper with a brief remark. By the methods presented here, we can prove results for the expected number of zeros $\EE N(T;\sigma)$ of
$$S_T(t;\sigma):=\sum_{n\leq T}\frac{X_n\cos(t\log n)}{n^\sigma}$$
in the interval $[T,2T]$.

In the case $\sigma>1/2$, one can show that $\EE N(T;\sigma)\ll T$, and for $0\leq \sigma \leq 1/2$, $\EE N(T;\sigma)\sim c_\sigma T\log T$, for some constant $c_\sigma>0$, and thus indicating a transition of this number of zeros. 

\noindent \textbf{Acknowledgements}. We are warmly thankful to Micah Millinovich for enlightening discussions that led us to Lemma \ref{lemma derivadas da zeta}. We also thank the referee for his/her suggestion and corrections that improved our exposition. This project is supported by CNPq - grant Universal 403037/2021-2, and by FAPEMIG - grant Universal APQ-00256-23.  

\bibliographystyle{siam}
\bibliography{ay.bib}

\begin{thebibliography}{10}

\bibitem{aymone_JLMS}
{\sc M.~Aymone, W.~Heap, and J.~Zhao}, {\em Partial sums of random
  multiplicative functions and extreme values of a model for the {R}iemann zeta
  function}, J. Lond. Math. Soc. (2), 103 (2021), pp.~1618--1642.

\bibitem{berndt_zeros_derivative}
{\sc B.~C. Berndt}, {\em The number of zeros for {$\zeta \sp{(k)}\,(s)$}}, J.
  London Math. Soc. (2), 2 (1970), pp.~577--580.

\bibitem{bourgade_RMT}
{\sc P.~Bourgade and M.~Yor}, {\em Random matrices and the {R}iemann zeta
  function}, in Journ\'{e}es \'{E}lie {C}artan 2006, 2007 et 2008, vol.~19 of
  Inst. \'{E}lie Cartan, Univ. Nancy, Nancy, 2009, pp.~25--40.

\bibitem{vu_repulsion}
{\sc Y.~Do, H.~Nguyen, and V.~Vu}, {\em Real roots of random polynomials:
  expectation and repulsion}, Proc. Lond. Math. Soc. (3), 111 (2015),
  pp.~1231--1260.

\bibitem{edelman_roots_random_poly}
{\sc A.~Edelman and E.~Kostlan}, {\em How many zeros of a random polynomial are
  real?}, Bull. Amer. Math. Soc. (N.S.), 32 (1995), pp.~1--37.

\bibitem{erdos_random_poly}
{\sc P.~Erd\"{o}s and A.~C. Offord}, {\em On the number of real roots of a
  random algebraic equation}, Proc. London Math. Soc. (3), 6 (1956),
  pp.~139--160.

\bibitem{garaev_real_part}
{\sc M.~Z. Garaev}, {\em On vertical zeros of {$\Re\zeta(s)$} and
  {$\Im\zeta(s)$}}, Acta Arith., 108 (2003), pp.~245--251.

\bibitem{irma_zeros_derivative}
{\sc F.~Ge and A.~I. Suriajaya}, {\em Note on the number of zeros of
  {$\zeta^{(k)}(s)$}}, Ramanujan J., 55 (2021), pp.~661--672.

\bibitem{ibragimov_random_poly_1}
{\sc I.~A. Ibragimov and N.~B. Maslova}, {\em The average number of real roots
  of random polynomials}, Dokl. Akad. Nauk SSSR, 199 (1971), pp.~13--16.

\bibitem{ibragimov_random_poly_2}
\leavevmode\vrule height 2pt depth -1.6pt width 23pt, {\em The mean number of
  real zeros of random polynomials. {II}. {C}oefficients with a nonzero mean},
  Teor. Verojatnost. i Primenen., 16 (1971), pp.~495--503.

\bibitem{ivic_livro}
{\sc A.~Ivi\'{c}}, {\em The {R}iemann zeta-function}, Dover Publications, Inc.,
  Mineola, NY, 2003.
\newblock Theory and applications, Reprint of the 1985 original [Wiley, New
  York; MR0792089 (87d:11062)].

\bibitem{kac_random_poly}
{\sc M.~Kac}, {\em On the average number of real roots of a random algebraic
  equation}, Bull. Amer. Math. Soc., 49 (1943), pp.~314--320.

\bibitem{Littlewood_random_pol_1}
{\sc J.~E. Littlewood and A.~C. Offord}, {\em On the {N}umber of {R}eal {R}oots
  of a {R}andom {A}lgebraic {E}quation}, J. London Math. Soc., 13 (1938),
  pp.~288--295.

\bibitem{Littlewood_random_pol_3}
\leavevmode\vrule height 2pt depth -1.6pt width 23pt, {\em On the number of
  real roots of a random algebraic equation. {III}}, Rec. Math. [Mat. Sbornik]
  N.S., 12(54) (1943), pp.~277--286.

\bibitem{littlewood_random_int_1}
\leavevmode\vrule height 2pt depth -1.6pt width 23pt, {\em On the distribution
  of the zeros and {$a$}-values of a random integral function. {I}}, J. London
  Math. Soc., 20 (1945), pp.~130--136.

\bibitem{littlewood_random_int_2}
\leavevmode\vrule height 2pt depth -1.6pt width 23pt, {\em On the distribution
  of zeros and {$a$}-values of a random integral function. {II}}, Ann. of Math.
  (2), 49 (1948), pp.~885--952; errata 50, 990--991 (1949).

\bibitem{montgomery_livro}
{\sc H.~L. Montgomery and R.~C. Vaughan}, {\em Multiplicative number theory.
  {I}. {C}lassical theory}, vol.~97 of Cambridge Studies in Advanced
  Mathematics, Cambridge University Press, Cambridge, 2007.

\bibitem{oanh_local_universality}
{\sc O.~Nguyen and V.~Vu}, {\em Roots of random functions: a framework for
  local universality}, Amer. J. Math., 144 (2022), pp.~1--74.

\bibitem{richert_zeta_bound}
{\sc H.-E. Richert}, {\em Zur {A}bsch\"{a}tzung der {R}iemannschen
  {Z}etafunktion in der {N}\"{a}he der {V}ertikalen {$\sigma =1$}}, Math. Ann.,
  169 (1967), pp.~97--101.

\bibitem{tenenbaumlivro}
{\sc G.~Tenenbaum}, {\em Introduction to analytic and probabilistic number
  theory}, vol.~163 of Graduate Studies in Mathematics, American Mathematical
  Society, Providence, RI, third~ed., 2015.

\bibitem{titchmarsh-theory-rzf}
{\sc E.~C. Titchmarsh}, {\em The theory of the {R}iemann zeta-function}, The
  Clarendon Press, Oxford University Press, New York, second~ed., 1986.
\newblock Edited and with a preface by D. R. Heath-Brown.

\bibitem{daodao_BLMS}
{\sc D.~Yang}, {\em Extreme values of derivatives of zeta and {$L$}-functions},
  Bull. Lond. Math. Soc., 56 (2024), pp.~79--95.

\end{thebibliography}

{\small{\sc \noindent
Departamento de Matem\'atica, Universidade Federal de Minas Gerais (UFMG), Av. Ant\^onio Carlos, 6627, CEP 31270-901, Belo Horizonte, MG, Brazil.} \\
\textit{Email address:} \verb|aymone.marco@gmail.com, caiomafiabueno@gmail.com| }
\end{document}